\documentclass[12pt]{amsart}
\usepackage{amssymb}
\usepackage{verbatim}
\usepackage{xspace}
\usepackage{graphicx}
\usepackage{graphics}
\usepackage{color}

\theoremstyle{plain}
\newtheorem{theorem}{Theorem}[section]

\newtheorem{question}[theorem]{Question}
\newtheorem{lemma}[theorem]{Lemma}

\theoremstyle{definition}

\newtheorem{construction}[theorem]{Construction}

\theoremstyle{remark}

\newtheorem{example}[theorem]{Example}

\newcommand{\bbC}{\mathbb C}

\newcommand{\bbN}{\mathbb N}

\newcommand{\bbR}{\mathbb R}

\newcommand{\la}{\lambda}  

\newcommand{\vpi}{\varphi}

     \newcommand{\sA}{\mathcal A}

     \newcommand{\sD}{\mathcal D}

\newcommand{\Cab}{C[a, b]}

\newcommand{\MnCab}{M_n(\Cab)}

\begin{document}

\title{Diagonalizing hermitian matrices of continuous functions}


\address{Department of Mathematics \\
        Iowa State University \\
        Ames, IA 50011}

\author[Cyr]{Justin Cyr}
\curraddr{Department of Mathematics \\
        Syracuse University \\
        215 Carnegie Building \\
        Syracuse, NY 13244-1150}
\email{jrcyr@syr.edu}

\author[Ekstrand]{Jason Ekstrand}
\email{ekstrand@iastate.edu}

\author[Meyers]{Nathan Meyers}
\curraddr{Department of Mathematics \\
        Bowdoin College \\
        Brunswick ME 04011}
\email{nathanlmeyers@gmail.com}

\author[Peoples]{Crystal Peoples}
\email{peoplesc@iastate.edu}

\author[Peters]{Justin Peters}
\email{peters@iastate.edu}

\begin{abstract} The problem of diagonalizing hermitian matrices of continuous fiunctions was studied by Grove and Pederson in 1984. While diagonalization is not possible in general,  in the presence
of differentiability conditions  we are able to obtain positive results in the case of $2\times 2$ matrices. It remains open whether our results can be extended to $n\times n$ matrices.
\end{abstract}



 \thanks{The authors acknowledge partial support from the National Science Foundation, DMS-0750986}

\maketitle

\section{Introduction}
In this note we consider a special case of the problem of diagonalizing hermitian matrices of continuous functions.
 We say that a hermitian element $A \in \MnCab$
is \emph{diagonalizable} if there is a unitary element $U \in \MnCab$ such that
$U^*AU$ is diagonal.  Thus, for each $t \in [a, b]$, $U(t)^*A(t)U(t)$ is
diagonal, and in particular each $U(t)$ is unitary.  We do not know who first
realized that hermitian elements of $\MnCab$ cannot, in general, be
diagonalized, but evidently this is folklore.  In 1983 Kadison (\cite{K})
proved that if $A \in M_n(L^{\infty}[a, b])$ is hermitian, then $A$ is
diagonalizable by a unitary in $M_n(L^{\infty}[a, b])$. In fact Kadison proved more generally that if $\sA$ is an abelian von Neumann algebra,
a hermitian element in $M_n(\sA)$ is diagonalizable by a unitary in $M_n(\sA).$

 However for $\MnCab$ the analogous statement is false.
(Example~\ref{E:nodiagonalization}) A year after Kadison's result appeared,
Grove and Pedersen proved that, for a large class of compact Hausdorff spaces $X$, if $A \in M_n(C(X))$ is hermitian with the property
that the eigenvalues of $A(x)$ are distinct for all $x \in X$, then $A$ is
diagonalizable in $M_n(C(X))$.   Their proof, however, is
an existence proof, and the question of how such a diagonalization can be
achieved is not addressed.  \cite{GP} also considers the question of which compact Hausdorff spaces $X$
have the property that every hermitian element $A \in M_n(C(X))$ is diagonalizable by a unitary in $M_n(C(X)).$

Here we consider the case where $X = [a, b],$ a compact interval.
Surprisingly, even for $2\times 2$ matrices it is not completely trivial to show that that if $A \in M_2(C[a, b])$ is hermitian with distinct eigenvalues,
then there is a unitary $U \in M_2(C[a, b]) $ such that $U^*AU $ is diagonal. In fact,  
for $2\times 2$ matrices we provide an algorithm
for diagonalization in case where the two eigenvalues are distinct.   

Our main result, Theorem~\ref{t:diag2}, is to obtain a diagonalization theorem for hermitian elements $A$ of $M_2(C[a, b])$ in which we relax the condition
that the eigenvalues of $A(t)$ must be distinct for all $t \in [a, b].$ To do this we suppose $A$ is continuously differentiable, and at those points $t$ where the
the eigenvalues of $A(t)$ coincide, the eigenvalues of $A'(t)$ are distinct.

This result gives rise to a variety of related questions which, as far as we are aware, have not been explored. The most obvious question is whether there is an analogue of
Theorem~\ref{t:diag2} for $\MnCab.$  Our proof uses the fact that one has an explicit formula for the eigenvalues in the $2\times 2$ case; thus an entirely different approach would be
required in the general case.
 Another question that could be asked is, if one assumes ``sufficient'' diffierentiability, then the conditions on the eigenvalues can be removed. Example~\ref{e:Cinftynotenough} shows that
even if the hermitian matrix has $C^{\infty}$ entries, it may not be diagonalizable. Finally we raise the question whether more can be done if the entries are analytic (\ref{q:analytic})

\section{Diagonalization} \label{s:diagonalization}

We begin with an example showing that continuous diagonalization is not always possible.
\begin{example} \label{E:nodiagonalization}
We work on $M_2(C[-1, 1]).$
Let $\chi$ be the characteristic function of the interval $[0, 1],$ and let
\[ A(t) =
t\left[
\begin{matrix}
1 &  \chi \\
\chi &  \chi
\end{matrix}
\right] .\]
Now the unitary matrix $U(t)  = I$ diagonalizes $A(t)$ for $t \in [-1, 0]$ and 
\[ U(t) = 
\frac{1}{\sqrt{2}}\left[
\begin{matrix}
1 & 1 \\
1 & -1
\end{matrix}
\right] \]
diagonalizes $A(t)$ for $t \in (0, 1].$
If $W(t)$ is another one parameter family of unitary matrices such that  $W(t)^*A(t)W(t) $ is diagonal, $\ t \in [-1, 1],$ then
 a calculation shows that  $W(t) = U(t) D(t)$ where
$D(t)$ is a diagonal unitary.  Thus, there is no choice of $W(t)$ which is continuous at $t = 0.$

\end{example}

\begin{theorem}\label{t:DiagonalizeHerm}
Let $A \in M_2(\Cab)$ such that $A(t)$ is Hermitian and $A(t)$ has distinct
eigenvalues for every $t \in [a, b]$. Then there exists $U \in M_2(\Cab)$
such that $U(t)$ is unitary and $U(t)^*A(t)U(t)$ is diagonal for all $t \in [a,
b]$. Furthermore, if $A$ consists of real-valued functions then $U$ can be
chosen to be real-valued.
\end{theorem}

It will be useful to introduce notation for $A(t).$
\begin{equation}
\label{e:def2x2matrix}
A(t)
= \begin{bmatrix}
f(t) & h(t)  \\
\bar h(t) & g(t)
\end{bmatrix}
\in M_2(\Cab)
\end{equation}
where $f,\ g$ are real-valued functions in $\Cab,$ and $h$ is complex-valued.
From the characteristic polynomial for $A(t)$ 
\begin{align}
p_{A(t)}(x)
&= \det[x I - A(t)] \notag \\
&= x^2 - (f(t) + g(t))x + f(t) g(t) - |h(t)|^2 \notag
\end{align}
 we obtain the two eigenfunctions
\begin{align}
\la^+(t)
&= \frac{f(t) + g(t) + \sqrt{(f(t) - g(t))^2 + 4|h(t)|^2}}{2}
        \label{e:ExplicitEval+} \\
\la^-(t)
&= \frac{f(t) + g(t) - \sqrt{(f(t) - g(t))^2 + 4|h(t)|^2}}{2}
        \label{e:ExplicitEval-}
\end{align}

Note that $\la^+(t) = \la^-(t)$ if and only if $A(t) = \la^+(t) I$.

\begin{lemma}
\label{l:hzeroiff}
Let $A(t)$ be as in (\ref{e:def2x2matrix}) and assume that $A(t)$ has distinct
eigenvalues for all $t \in [a, b]$. Let $\la(t)$ be an eigenvalue of $A(t)$ for
every $t \in [a, b]$. For $t \in [a,b]$, the following are equivalent:
\begin{enumerate}
\item $h(t) = 0$,
\item $f(t) - \lambda(t) = 0$ or $g(t) - \lambda(t) = 0$,
\item either $f(t) - \lambda(t) = 0$ or $g(t) - \lambda(t) = 0$, but not both.
\end{enumerate}
\end{lemma} 
\begin{proof}
If $h(t) = 0$ then $A(t)$ is diagonal, and since $\la(t)$ is an eigenvalue for
$A(t)$, $f(t) - \lambda(t) = 0$ or $g(t) - \lambda(t) = 0$ but not both since
$A(t)$ has distinct eigenvalues. If $f(t) - \la(t) = 0$ or $g(t) - \la(t) = 0$
then by equations (\ref{e:ExplicitEval+}) and (\ref{e:ExplicitEval-}), $|h(t)|^2 = 0$ and so $h = 0$.
\end{proof}

Let $\la$ refer either to $\la^+$ or $\la^-.$  We wish to find ${\bf v}(t) = \begin{bmatrix} v_1(t) \\ v_2(t) \end{bmatrix}$
continuous and non-vanishing such that $A(t){\bf v}(t) = \la(t){\bf v}(t)$.
Equivalently,
\begin{align}
(f(t) - \la(t))v_1(t) + h(t) v_2(t) &= 0 \label{e:2x2Evec1}\\
\bar h(t) v_1(t) + (g(t) - \la(t))v_2(t) &= 0 \label{e:2x2Evec2}
\end{align}
along with the condition that $v_1(t)$ and $v_2(t)$ are never simultaneously
zero. The following construction allows us to do this explicitly.

\begin{construction} \label{c:EigenVecConstruction} Let $A \in M_2(C[a, b]$ have distinct eigenvalues for all $t \in [a, b],$ 
and fix $\la$ to be either $\la^+$ or $\la^-$ throughout the construction.
Input: $n \in \bbN$, $t_n \in [a, b)$, and ${\bf v}(t_n) \in \bbC^2$ such that
${\bf v}(t_n) \ne 0$ and $A(t_n){\bf v}(t_n) = \la(t_n){\bf v}(t_n)$.

By Lemma \ref{l:hzeroiff}, we cannot have $\la(t_n) = f(t_n) = g(t_n). $
Suppose $f(t_n) \ne \la(t_n)$. Then, by continuity, $f -
\la$ is nonzero in a neighborhood of $t_n$. Let 
\[
t_{n + 1} = \min (\{t \in [t_n, b] : f(t) = \la(t)\} \cup \{b\})
\text{.}
\]
On $[t_n, t_{n + 1})$, define
\begin{equation} \label{e:EigenVecDef1}
{\bf u}(t) =
\begin{bmatrix}
u_1(t) \\
u_2(t)
\end{bmatrix}
= \begin{bmatrix}
\frac{h(t)}{\la(t) - f(t)} \\
1
\end{bmatrix}
\text{.}
\end{equation}
Then $u_1$ and $u_2$ satisfy (\ref{e:2x2Evec1}) on $[t_n, t_{n + 1})$. Since
$A(t) - \la(t) I$ is singular, (\ref{e:2x2Evec1}) and (\ref{e:2x2Evec2}) form a
dependent set of equations, and so (\ref{e:2x2Evec2}) is satisfied. Therefore,
${\bf u}(t)$ is an eigenvector of $A(t)$ for $\la(t)$ for all $t \in [t_n, t_{n
+ 1})$. Since $A(t)$ has distinct eigenvalues,  the geometric multiplicity is one and there
is some constant $c_1 \in \bbC$ with $c_1 \ne 0$ such that ${\bf v}(t_n) =
c_1{\bf u}(t_n)$.

If $\la \ne f$ on all of $[t_n, b]$, then $t_{n + 1} = b$ and we can 
extend ${\bf u}(t)$ to $[t_n, t_{n + 1}]$ using (\ref{e:EigenVecDef1}). Then let ${\bf
v}(t) = c_1{\bf u}(t)$ on $[t_n, b]$ and we are done.

Otherwise, $\la(t_{n + 1}) = f(t_{n + 1})$. Then, by Lemma \ref{l:hzeroiff},
$\la(t_{n + 1}) \ne g(t_{n + 1})$ and by continuity there is some $t_n <
\alpha < t_{n + 1}$ such that $\la \ne f$ on $[\alpha, t_{n + 1}]$. On
$[\alpha, t_{n + 1}]$, define
\begin{equation} \label{e:EigenVecDef2}
{\bf w}(t) =
\begin{bmatrix}
w_1(t) \\
w_2(t)
\end{bmatrix}
= \begin{bmatrix}
1 \\
\frac{\bar h(t)}{\la(t) - g(t)}
\end{bmatrix}
\text{.}
\end{equation}
By the same argument as above, ${\bf w}(t)$ is an eigenvector of $A(t)$ for
$\la(t)$ for all $t \in [\alpha, t_{n + 1}]$. Furthermore, since $\la(\alpha)$
has multiplicity one, there is some nonzero $c_2 \in \bbC$ such that ${\bf
u}(\alpha) = c_2 {\bf w}(\alpha)$. Then, let ${\bf v} = c_1 {\bf u}$ on $[t_n,
\alpha]$ and let ${\bf v} = c_1 c_2 {\bf w}$ on $[\alpha, t_{n + 1}]$. Then we have
 that ${\bf v}$ is continuous and non-vanishing on $[t_n, t_{n +
1}]$ and that $A(t){\bf v} (t) = \la(t) {\bf v}(t)$ for all $t \in [t_n, t_{n +
1}]$.

Finally, in the case where $f(t_n) = \la(t_n)$, we reverse the roles of (\ref{e:EigenVecDef1}) and (\ref{e:EigenVecDef2}).
\end{construction}

\begin{lemma}\label{l:EigenVecConstructionTerm}
For $A \in M_2(\Cab)$, construction \ref{c:EigenVecConstruction} terminates
after finitely many steps.
\end{lemma}
\begin{proof}
Let $A \in M_2(\Cab)$ and let $t_0 = a$. For all iterations of the construction we make the same choice of
$\la$ as either $\la^+$ or $\la^-.$  Suppose that construction
\ref{c:EigenVecConstruction} never terminates and yields the sequence
$\{t_n\}$. Since $\{t_n\}$ is increasing and $t_n < b$ for every $n, \ t_n$ has a limit, $c,\ a < c \leq b$. Observe that, for
each $n > 0$, $f(t_n) = \la(t_n)$ or $g(t_n) = \la(t_n)$ by our choice of the
$t_n$. Further observe that if $f(t_n) = \la(t_n)$ then $g(t_{n + 1}) =
\la(t_{n + 1})$ and vice-versa because of the way we choose the $t_n$ and the
assumption that $t_n < b$. Therefore, since $f$, $g$, and $\la$ are continuous,
$\lim_{n \to \infty} f(t_n) - \la(t_n) = 0$ and $\lim_{n \to \infty} g(t_n) -
\la(t_n) = 0$, and so $f(c) = g(c) = \la(c)$. However,
this contradicts Lemma \ref{l:hzeroiff}, and so construction
\ref{c:EigenVecConstruction} must terminate after only finitely many steps.
\end{proof}

Finally, we can prove Theorem \ref{t:DiagonalizeHerm}.

\begin{proof}[Proof of Theorem \ref{t:DiagonalizeHerm}]
Let $A \in M_2(\Cab)$ and let $t_0 = a$. Let ${\bf v(t_0)}$ be the eigenvector
corresponding to $\la^+(t_0)$ for $A(t_0)$. Then applying construction
\ref{c:EigenVecConstruction} and Lemma \ref{l:EigenVecConstructionTerm} yields
${\bf v}_1(t):[a, b] \to \bbC^2$ such that $A(t){\bf v}_1(t) = \la^+(t){\bf
v}_1(t)$ for $t \in [a, b]$.  Similarly, we can get ${\bf v}_2(t):[a, b] \to
\bbC^2$ such that $A(t){\bf v}_2(t) = \la^-(t){\bf v}_1(t)$ for $t \in [a, b]$.

Let ${\bf u}_j(t) = \frac{{\bf v}_j(t)}{\|{\bf v}_j(t)\|_2},\ j = 1, \ 2$. Then, for all $t \in [a, b]$, we have
 that, since $A(t)$ is hermitian and $\la^+(t) \ne \la^-(t)$, ${\bf
v}_1(t) \perp {\bf v}_2(t)$. Therefore, letting $U \in M_2(C[a, b])$ be defined by
$U(t) = [{\bf u}_1(t) ~ {\bf u}_2(t)]$, it follows that $U(t)$ is
unitary and $U(t)^*A(t)U(t)$ is diagonal for all $t$.

Finally, it may be observed that when $h(t)$ is real-valued, and the initial
${\bf v}(t_0)$ is real, construction \ref{c:EigenVecConstruction} yields a
real-valued result. Therefore, if $A$ is real-valued, $U$ will be real-valued.
\end{proof}

We now relax the assumption that the eigenvalues of $A(t)$ are distinct for all $t.$

\begin{theorem} \label{t:diag2}
Let $A \in M_2(\Cab)$ be continuously differentiable and suppose that $A(t)$
has distinct eigenvalues except possibly at finitely many points $t $ and at
those points $A'(t)$ has distinct eigenvalues. Then there exists a unitary
matrix $U \in M_2(\Cab)$ such that $U^* A U$ is diagonal.
\end{theorem}

We begin with a lemma.

\begin{lemma} \label{l:TwoNormDiff}
Let ${\bf v}:[a, b] \to \bbR^n$ be continuously differentiable. Suppose that
${\bf v}(t) = 0$ for only finitely many points $t $. Then there exists $\mu:[a,
b] \to \mathbb{R}$ continuously differentiable such that $|\mu(t)| = \|{\bf
v}(t)\|_2$.
\end{lemma}
\begin{proof}

Suppose there is at least one point in the open interval at which ${\bf v} $ vanishes.
Let $a < t_1 < t_2 < \ldots < t_s < b$ be those
points in $(a, b)$ where ${\bf v}(t_i) = 0$ and, for each $1 \le i \le s+1$,
define
\[
\mu(t)
= (-1)^i \|{\bf v}(t)\|_2
\]
on $[t_{i - 1}, t_i],$ with $t_0 = a$ and $t_{s+1} = b.$

Fix $1 \le i \le s$. Since ${\bf v}(t)$ is continuously differentiable, we
may apply Taylor's theorem to each component of ${\bf v}(t)$ which yields
\[
{\bf v}(t) = {\bf v}(t_i) + {\bf v}'(t_i)(t - t_i) + {\bf h}(t)(t - t_i)
\]
where ${\bf h}(t) \to {\bf 0}$ as $t \to t_i$. By assumption, ${\bf v}(t_i) =
0$, so
\[
{\bf v}(t)
= {\bf v}'(t_i)(t - t_i) + {\bf h}(t)(t - t_i)
= (t - t_i)\big({\bf v}'(t_i) + {\bf h}(t)\big)
\text{.}
\]
Then, for $t \ne t_i$,
\begin{align*}
\frac{d}{dt} \|{\bf v}(t)\|_2
&= \frac{d}{dt} \sqrt{{\bf v}(t) \cdot {\bf v}(t)} \\
&= \frac{{\bf v}(t) \cdot {\bf v}'(t)}{\|{\bf v(t)}\|_2} \\
&= \frac{(t - t_i)\big({\bf v}'(t_i) + {\bf h}(t)\big) \cdot {\bf v}'(t)}{
        \|(t - t_i)\big({\bf v}'(t_i) + {\bf h}(t)\big)\|_2}
\text{.}
\end{align*}
Since ${\bf v}'(t)$ is continuous, if ${\bf v}'(t_i) \ne 0$,
\[
\lim_{t \to t_i^+} \frac{d}{dt} \|{\bf v}(t)\|_2
= \lim_{t \to t_i^+} \frac{(t - t_i){\bf v}'(t_i) \cdot {\bf v}'(t)}{
        |(t - t_i)|\|{\bf v}'(t_i)\|_2}
= \frac{{\bf v}'(t_i) \cdot {\bf v}'(t_i)}{\|{\bf v}'(t_i)\|_2}
= \|{\bf v}'(t_i)\|_2
\]
and if ${\bf v}'(t_i) = 0$,
\[
\lim_{t \to t_i^+} \frac{d}{dt} \|{\bf v}(t)\|_2
= \lim_{t \to t_i^+} \frac{t - t_i}{|t - t_i|}
        \frac{{\bf h}(t)}{\|{\bf h}(t)\|} \cdot {\bf v}'(t)
= 0
= \|{\bf v}'(t_i)\|_2
\]
since  ${\bf v}'(t)
\to 0$ as $t \to t_i$.

Similarly,
\[
\lim_{t \to t_i^-} \frac{d}{dt} \|{\bf v}(t)\|_2
= -\|{\bf v}'(t_i)\|_2
\]
Therefore, by our definition
of $\mu(t)$ and this difference of sign, $\mu'(t)$ is well defined and
continuous at $t_i$.

If ${\bf v}$ never vanishes on $(a, b),$ we can take $\mu(t) = ||{\bf v}(t)||.$
 Note that it is irrelevant whether ${\bf v}$ vanishes at one or both endpoints.

\end{proof}

\begin{proof}[Proof of theorem \ref{t:diag2}.]
Let $h(t) = h_r(t) + ih_c(t)$ where $h_r$ and $h_c$ are real-valued.
 For  $t \in [a, b]$, the eigenvalues of $A(t)$ are as in equations
(\ref{e:ExplicitEval+}) and (\ref{e:ExplicitEval-}).

Let $Z \subseteq [a, b]$ be the set of points $t$ at which $A(t)$ does not have
distinct eigenvalues; then $|Z| < \infty$ by hypothesis.  Let 
\[
{\bf v}(t) = [f(t) - g(t), 2h_r(t), 2h_c(t)]^T
\]
and observe that ${\bf v}(t) = 0$ exactly on $Z$. Applying Lemma
\ref{l:TwoNormDiff} to ${\bf v(t)}$, there is some continuously differentiable
$\mu:[a, b] \to \mathbb{R}$ such that
\[
|\mu(t)|
= \|{\bf v}(t)\|_2
= \sqrt{(f(t) - g(t))^2 + 4(h_r^2(t) + h_c^2(t))}
\text{.}
\]
Define $\la(t) = \frac{1}{2} [f(t) + g(t) + \mu(t)]$; then $\la(t)$ is
continuously differentiable and $\la(t)$ is an eigenvalue for $A(t)$ for all $t
\in [a, b]$. Let $B(t) = A(t) - \la(t) I$. Then $B(t)$ is continuously
differentiable, $B \ne 0$ on $[a, b] \setminus Z$ and $B'(t) \ne 0$ for $t \in
Z$. Applying Lemma \ref{l:TwoNormDiff} to $B(t)$, there exists $\tau:[a, b]
\to \mathbb{R}$ continuously differentiable such that $|\tau(t)| = \|B(t)\|_2$
for all $t \in [a, b]$ (where $B(t)$ is treated as a vector with $4$
entries). Next define
\[
C(t) = \frac{1}{\tau(t)} B(t)
\]
and observe that, by l'H\^opital's rule, we can define $C(t) =
\frac{B'(t)}{\tau'(t)}$ for $t \in Z$ and $C(t)$ will be continuous on $[a,
b]$. Since $\det:M_2(\mathbb{C}) \to \mathbb{C}$ is continuous and $\det(C(t))
= 0$ for all $t \not \in Z$, $\det(C(t)) \equiv 0$ on $[a, b]$. Furthermore,
for any $t \in [a, b]$, if $t \not \in Z$ then $C(t)$  has distinct
eigenvalues and if $t \in Z$, then $C(t) = \frac{B'(t)}{\tau'(t)}$ which has
distinct eigenvalues. Therefore, $C(t)$ is well-defined with distinct
eigenvalues and so by Theorem \ref{t:DiagonalizeHerm} we can find $U \in
M_2(\Cab)$ such that $U(t)$ is unitary and $U(t)^*C(t)U(t)$ is diagonal for all
$t \in [a, b]$. Then
\begin{align*}
U(t)^*A(t)U(t)
&= U(t)^*(B(t) + \la(t) I)U(t) \\
&= U(t)^*B(t)U(t) + \la(t) I \\
&= \tau(t)U(t)^*C(t)U(t) + \la(t) I
\end{align*}
is diagonal.
\end{proof}

\begin{example} \label{e:Cinftynotenough}
Let
\[ \vpi(t) =
\begin{cases}
\exp(-1/t^2) \text{ if } t \neq 0 \\
0 \text{ if } t = 0
\end{cases} \]
and let
\[ \psi(t) =
\begin{cases}
0 \text{ if } t \leq 0 \\
\vpi(t) \text{ if } t > 0
\end{cases} .\]
Note that both $\vpi$ and $\psi$ are C$^{\infty}$ functions, hence the matrix

\[
A(t) = 
\begin{bmatrix}
\vpi(t) & \psi(t) \\
\psi(t) & \psi(t) \\
\end{bmatrix}
=
\vpi(t) \left[
\begin{matrix}
1 &  \chi \\
\chi &  \chi
\end{matrix}
\right] 
\]
is  C$^{\infty}$ on $[-1, 1]$ , where $\chi$ is the characteristic function of $[0, 1].$  Note that $A(t)$ has repeated eigenvalues only at $t = 0$.

Now $0$ is an eigenvalue for $A(t)$ for every $t$. However,
$\begin{bmatrix} 0 \\ 1 \end{bmatrix}$ is an eigenvector for $0$ for $t \in [-1,
0)$ and $\begin{bmatrix} 1 \\ -1 \end{bmatrix}$ is an eigenvector for $0$ for $t
\in [0, 1]$. Therefore, $A(t)$ cannot be continuously diagonalized.
\end{example}

Since differentiability alone is not sufficient to guarantee diagonalizability of hermitian elements, one can ask
if analyticity suffices.

\begin{question} \label{q:analytic}
Let $A \in \MnCab$ be hermitian. Writing the entries of $A$ as in (\ref{e:def2x2matrix}), assume there is a domain $\sD \subset \bbC$ whch contains $[a, b]$
and that $f,\ g, \ h$ are analytic in $\sD$ with $f, \ g$ real on $[a, b].$ Is $A$ diagonalizable over $[a, b] ?$

\end{question}

\end{document}